\documentclass[12pt]{amsart}
\usepackage[hidelinks]{hyperref}
\usepackage{tikz}

\usepackage[utf8]{inputenc}
\usepackage{amsmath}
\usepackage{amssymb} 
\usepackage[numbers]{natbib}
\usepackage{prettyref}
\usepackage{tikz}
\usepackage{tikz-cd}

\theoremstyle{plain}
\makeatletter
\newtheorem*{thm@P}{Theorem \pipp@}

\makeatother

\numberwithin{equation}{section}

\newtheorem{thm}{Theorem}[section]
\newtheorem*{thm*}{Theorem}
\newrefformat{thm}{Theorem~\ref{#1}} 
\makeatletter 
\let\old@newtheorem\newtheorem 
\renewcommand{\newtheorem}[2]{\old@newtheorem{#1}[thm]{#2} 
	\newrefformat{#1}{#2~\ref{##1}}} 
\makeatother 

\newtheorem{prop}{Proposition}
\newtheorem{cor}{Corollary}

\newtheorem{lemma}{Lemma}
\newtheorem{fact}{Fact}

\theoremstyle{definition}

\newtheorem{defn}{Definition}

\newtheorem{example}{Example}

\theoremstyle{remark}
\newtheorem{rem}{Remark}

\newcommand{\Q}{\mathbb Q}
\newcommand{\R}{\mathbb R}
\newcommand{\N}{\mathbb N}
\newcommand{\K}{\mathbb K}

\newcommand{\no}{\mathbf {No}}
\newcommand{\Z}{\mathbb Z}

\newcommand{\n}{\mathfrak n}
\newcommand{\m}{\mathfrak m}
\newcommand{\x}{\mathbf{x}}

\newcommand{\rest}{\! \restriction \!}

\DeclareMathOperator{\supp}{supp}

\newcommand{\M}{\mathfrak{M}}
\newcommand{\fN}{\mathfrak{N}}

\newcommand{\T}{\mathbb{T}}

\newcommand\ntrunc{\mathrel{\blacktriangleleft}}
\newcommand\ntrunceq{\mathrel{\ooalign{{\raise-1ex\hbox{$\relbar$}}\cr\raise.1ex\hbox{$\ntrunc$}}}}

\def\M{\mathfrak{M}}

\def\J{\mathbb{J}}

\def\lm{\mathrm{lm}} 
\def\lc{\mathrm{lc}} 
\def\lt{\mathrm{lt}} 
\def\level{\mathrm{lv}} 
\def\Lhp{(\!(} 
\def\Rhp{)\!)} 
\def\k{\mathbf{k}}
\def\grid{\mathrm{g}}

\newcommand{\InfP}[1]{#1^{\uparrow}}
\newcommand{\ReP}[1]{#1^{\circ}}
\newcommand{\EpsP}[1]{#1^{\downarrow}}

\newenvironment{pf}{\begin{proof}}{\end{proof}}

\title{On the value group of the transseries}
\author{Alessandro Berarducci}
\author{Pietro Freni}
\date{December 11, 2020. ArXiv version January 3, 2021}
\thanks{The first author was partially supported by the Italian research project PRIN 2017, ``Mathematical logic: models, sets, computability'', Prot. 2017NWTM8RPRIN.}
\address{Dipartimento di Matematica, Università di Pisa, Largo Bruno Pontecorvo 5, 56127 Pisa, Italy}
\address{School of Mathematics, University of Leeds, Leeds
	LS2 9JT, United Kingdom}
\begin{document}
\maketitle
\begin{abstract}
	We prove that the value group of the field of transseries is isomorphic to the additive reduct of the field. 
\end{abstract}
\tableofcontents

\section{Introduction}

Given a real closed field $(K, +, \cdot, <)$, the possibility of defining an ordered exponential on it, that is, a ordered group isomorphism $\exp : (K, +, <) \to (K^{>0}, \cdot <)$, is strictly connected to the properties of its natural valuation (i.e.\ the valutation whose valuation ring is given by the elements bounded in absolute value by some natural number). For example if an exponential exists, then the valuation group $v(K)$ is isomorphic to an additive complement of the valuation ring. In \cite{Kuhlmann1997} it is shown that for ordered fields that are maximal with respect to their natural valuation this condition fails unless the value group is trivial (in which case $K \subseteq \R$), so maximal non-archimedean ordered fields do not admit an exponential.
In \cite{Kuhlmann2005} the same property is used, the other way around, to show that, given a regular uncountable cardinal $\kappa$, there is a non-trivial group $\M$ such that the field $K=\R\Lhp\M\Rhp_\kappa \subseteq \R\Lhp\M\Rhp$ of $\kappa$-bounded generalized series admits an exponential map.

In \cite{Berarducci2018b} the authors study the related property of admitting an isomorphism $\Omega : (K, +,<) \to (\M, +, <)$, where $\M \subseteq K^{>0}$ is an embedded multiplicative copy of the value group. They call such an isomorphism  \emph{omega-map}, in analogy to Conway's omega map on Surreal numbers \cite{Conway76}, and prove that, for fields of the form  $\R\Lhp\M\Rhp_\kappa$, its existence implies the existence of an exponential map. Moreover, for fields of the same form, the converse holds under the additional hypothesis that the value group $\M$ is order isomorphic to its positive cone $\M^{> 1}$.

The above results leave open the question whether the field $\T$ of LE-transseries \cite{DriesMM2001} admits an omega-map, as it is not a field of the form $\R\Lhp \M\Rhp_\kappa$.
Besides $\R$, transseries are arguably the most important example of an exponential field. They are an important tool in asymptotic analysis and have been used by Ecalle \cite{Ecalle1992} to give a positive solution to Dulac's conjecture (the finiteness of limit cycles in planar polynomial vector fields).

The main result of the paper is that the value group of $\T$ is isomorphic to $\T$ itself as an ordered additive group. This is proved by explicitly constructing an omega-map $\Omega:\T \to \M^{LE}$ where $\M^{LE}$ is the group of transmonomials. 

Finally, abstracting some tools which have been useful in the construction we generalize results in \cite{Berarducci2018b} to a much wider class of fields. 
\medskip

We describe below the main ideas of the paper. 
In the case of exponential fields of the form $K=\R\Lhp \M \Rhp_\kappa$ treated in \cite{Berarducci2018b}, any order isomorphism $\eta : \M\to \M^{> 1}$ naturally induces an isomorphism $H: K \to \R\Lhp \M^{> 1}\Rhp$ of ordered $\R$-vector spaces from the field to the canonical complement of its valuation ring. This in turn induces an omega-map. 

When considering the case of $\T$, we follow a similar approach but there are many additional complications. We recall that $\T$ is a subfield of $\R \Lhp\M^{LE}\Rhp$. Any order isomorphism $\M^{LE} \to \M^{LE,> 1}$ induces an embedding of ordered $\R$-vector spaces $\T \to \R\Lhp \M^{LE,> 1}\Rhp$, however there is no guarantee that its image is $\T^\uparrow := \T\cap \R\Lhp \M^{LE,> 1}\Rhp$. In this paper we show that $\M^{LE}$ is order isomorphic to $\M^{LE,>1}$ and we produce a particular order isomorphism $\eta:\M^{LE}\to \M^{LE,>1}$ which induces an isomorphism of ordered $\R$-vector spaces $\T\to \T^\uparrow$. This in turn will induce an omega-map (Theorem \ref{thm:Main}).
%
%
%

In the above process we are lead to consider an ideal of subsets of $\M^{LE}$ (generated by subgroups) and characterize $\T$ as the field of series in $\R\Lhp\M^{LE}\Rhp$ with support in the ideal (Section ~\ref{sec:Summability}). A subset of $\M^{LE}$ will be called bounded if it belongs to the ideal and a map  $\eta: \M^{LE}\to \M^{LE,>1}$ will be called bounded if it maps bounded sets to bounded sets. The proof of the main result is then reduced to the problem of constructing a bounded order isomorphism $\eta:\M^{LE}\to \M^{LE,>1}$ with bounded inverse. This is achieved in Lemma \ref{Important1}.

\smallskip 

Endowing groups of monomials with suitable ideals of subsets  yields flexible constructions of many intereresting fields, encompassing $\T$, the $\kappa$-bounded series and the field of Puiseux series (Definition~\ref{defn:BField}).
%
%
  Some results of \cite{Berarducci2018b} generalize easily to this new setting (see Theorem~\ref{Omega&Exp}).

For simplicity of notation throughout the paper we work with fields whose residue field is $\R$, but with minor modifications we could have taken any other model of the first order theory of $(\R,\exp)$. 

\section{Hahn fields}

Given a multiplicatively written ordered abelian group $\M$ denote by $\R\Lhp \M\Rhp$ the field of Hahn's generalized series (cfr.\ \cite{Hahn1907}) with monomials from $\M$, real coefficients and reverse-well-ordered support
\[ \R\Lhp \M\Rhp = \{ f  \in \R^{\M} :  \supp (f)\; \text{is reverse-well-ordered}\}\]
where 
$$\supp(f)= \{ \mathfrak{m} \in \M : f(\mathfrak{m}) \neq 0\}$$
denotes the support of the series. The value of $f$ at $\mathfrak{m}$ is  referred to as the \emph{coefficient of the monomial $\mathfrak{m}$ in the series $f$} or the \emph{coefficient of $f$ at $\mathfrak{m}$} and written as $f_{\mathfrak{m}}:= f(\mathfrak{m})$.

The set $\R\Lhp \M \Rhp$ is naturally an ordered field extension of $\R$. Sums and multiplication by scalars in $\R$ are defined termwise and order is lexicographic, or equivalently $f >0$ if and only if $f_{\max \supp (f)}>0$. Multiplication has a Cauchy-like definition
\[ (fg)_\mathfrak{m}= \sum_{\mathfrak{n}\mathfrak{o} = \mathfrak{m}} f_{\mathfrak{n}}g_{\mathfrak{o}}\]
and the fact that the supports of $f$ and $g$ are reverse well ordered ensures that the sum on the right hand side has only finitely many non-zero terms and that the support $\supp(fg) \subseteq \supp(f) \supp(g)$ is still reverse-well-ordered.

Given a set $I$, an $I$-indexed family of series $(f_i \in \R\Lhp \M \Rhp : i \in I)$ is said to be \emph{summable} if the union of the supports of its elements $\bigcup_{i\in I} \supp(f_i)$ is reverse well ordered and for every $\mathfrak{m} \in \M$ the set $\{i : f_{i, \mathfrak{m}} \neq 0\}$ is finite: in such a case the formal sum is defined as
\[ \sum_{i \in I} f_i := g \quad \text{where} \quad g_{\mathfrak{m}}:= \sum_{\mathfrak{m} \in \supp(f_i)} f_{i, \mathfrak{m}}.\]
A monomial $\mathfrak{m} \in \M$ is usually identified with the series having coefficient $1$ at $\mathfrak{m}$ and $0$ at other monomials. A \emph{term} is then a series of the form $t=k \mathfrak{m}$ with $k \in \R$.
With such a convention, for every reverse well ordered set $S$ of monomials, a family of terms of the form $(k_{\mathfrak{m}} \mathfrak{m} : \mathfrak{m} \in S)$ is summable: elements of $\R\Lhp \M \Rhp$ are thus usually written as $\sum_{\mathfrak{m} \in S} f_\mathfrak{m} \mathfrak{m}$ for some reverse well ordered $S\subseteq \mathfrak{M}$ or as $\sum_{i<\alpha} k_i \mathfrak{m}_i$ where $\alpha$ is an ordinal number and $(\mathfrak{m}_i : i <\alpha)$ is a strictly decreasing $\alpha$-sequence in $\M$.\\

For every element $f$ of $\R\Lhp \M \Rhp\setminus \{0\}$ it makes sense to talk about the \emph{leading} monomial, coefficient and term of $f$, which we denote respectively as
\[\lm(f):=\max \supp (f), \quad \lc(f):=f_{\lm(f)}, \quad \lt(f):=\lm(f)\lc(f).\]
Since $\R\Lhp \M \Rhp$ is an ordered field it makes sense to define dominance and archimedean equivalence on non-zero elements. These notions have an easy characterization in terms of the leading term of a series.

\begin{defn}
For $x, y \in \R\Lhp \M \Rhp \setminus \{0\}$, we have the following relations
\begin{enumerate}
\item \emph{dominance:} $x \preceq y \;\Longleftrightarrow\; \exists n \in \N,\;  |x| < n|y| \;\Longleftrightarrow\; \lm(x) \le \lm(y)$;
\item \emph{comparability:} $x \asymp y \;\Longleftrightarrow\; \big( x \preceq y \; \& \; y \preceq x\big) \;\Longleftrightarrow\; \lm(x) =\lm(y)$;
\item \emph{strict dominance:} $x \prec y \;\Longleftrightarrow\; \big( x \preceq y \; \& \; x \not\asymp y \big) \;\Longleftrightarrow\; \lm(x) < \lm(y)$;
\item \emph{asymptotic equivalence:} $x \sim y \;\Longleftrightarrow\; x-y \prec x \;\Longleftrightarrow\; \lt(x)=\lt(y)$.
\end{enumerate}
Series that are $\succ 1$ are said to be \emph{infinite} whereas elements $\prec 1$ will be said to be \emph{infinitesimal}.
\end{defn}

\begin{rem}
The function $\lm : \mathbb{R}\Lhp \M \Rhp\setminus \{0\} \rightarrow \M$, where we endow $\M$ with the opposite order, is a field valuation whose residue field is $\R$ and whose valuation ring is the set of series $f$ such that $\supp(f) \le 1$. It is called \emph{archimedean valuation}. 
\end{rem}

\begin{rem}
We recall that an extension of valued fields is immediate if it preserves the value group and residue field. It is worth mentioning that $\R\Lhp \M\Rhp$ is maximal in the sense that it has no proper immediate extension (see \cite[p.193, Satz~26]{Krull1932}).  Thus every proper ordered field extension $F\supseteq \R\Lhp \M\Rhp$ has non-zero elements that are not comparable to any element of $\R\Lhp\M\Rhp$. 
\end{rem}

\begin{rem}
Every element of  $f = \sum_{i < \alpha} k_i\mathfrak{m}_i$ of $\R\Lhp \mathfrak{M} \Rhp$ decomposes uniquely as
\[ f= \InfP{f} + \ReP{f} + \EpsP{f}\]
where $\supp (\InfP{f}) > 1$, $ \ReP{f} \in \R$ and $\EpsP{f} \prec 1$, whereas every non-zero element $f$ decomposes multiplicatively as
\[ f= \lm(f) \lc(f) \left( 1 + \frac{f-\lt(f)}{\lt(f)}\right)\]
where $\lm(f) \in \M$, $\lc(f)\in \R\setminus \{0\}$ and $\frac{f-\lt(f)}{\lt(f)} \prec 1$.
\end{rem}

\begin{rem}
It may worth remarking that in the definitions above one could have considered any ordered field $\k$ instead of $\R$: in this case instead of $\preceq$ one must consider the $\k$-dominance relation defined as $x \preceq_\k y \Leftrightarrow \exists k \in \k, \; |x| < k |y|$.
\end{rem}

If $(\M, \cdot, 1, \prec)$ is an orderd group we will denote by $\M^{\succ 1}$ the set of elements of $\mathfrak{M}$ greater than $1$.

\begin{fact}\label{Summability1}
If $\mathfrak{N}$ is a multiplicatively written ordered abelian group, $\k$ is a subfield of $\mathbb{R} \Lhp \mathfrak{N} \Rhp$ and $\M\subseteq \mathfrak{N}$ is a subgroup such that $\M^{\succ 1}> \k$, then for every ordinal $\alpha$, every strictly decreasing sequence $(\mathfrak{m}_i)_{i <\alpha}$ of elements of $\M$, and every sequence $(k_i)_{i <\alpha}$ of elements of $\k$, the family $k_i \mathfrak{m}_i$ is summable.
\begin{proof}
Just note that $\supp(k_i \mathfrak{m}_i) < \supp(k_j \mathfrak{m}_j)$ for $i > j$.
\end{proof}
\end{fact}

\begin{defn}\label{GenCoeffSums}
In the hypothesis of Fact~\ref{Summability1} we denote the set of elements of the form $\sum_{i <\alpha} k_i \mathfrak{m}_i$ with $k_i \in \k$ and $\mathfrak{m}_i \in \M$ as $\k\Lhp \M\Rhp$: this is a subfield of $\R\Lhp \mathfrak{N} \Rhp$ isomorphic to the field of generalized series with coefficients from $\k$ and monomials from $\M$.

More generally if $\Gamma$ is a subset of $\mathfrak{N}$, $\k$ is a subfield of $\R\Lhp \mathfrak{N}\Rhp$, and $\k < \langle \Gamma \rangle^{\succ 1}$, where $\langle \Gamma \rangle$ is the subgroup generated by $\Gamma$, then $\k\Lhp \Gamma\Rhp$ will denote the $\k$-vector subspace of $\R\Lhp \mathfrak{N} \Rhp$ consisting of the series of the form $\sum_{i <\alpha} k_i \mathfrak{m}_i$ with all $\mathfrak{m}_i$ laying in $\Gamma$ and $k_i$ laying in $\k$.

For example if $\k=\R$ then every $\Gamma \subseteq \mathfrak{N}$ satisfies the hypothesis, so given any $\Gamma \subseteq \mathfrak{N}$, $\R\Lhp\Gamma\Rhp$ is the set of series with support included in $\Gamma$.
\end{defn}

\begin{rem}[{\cite[Cor. 2.19]{Berarducci2018b}  or \cite[\S 1.4]{DriesMM2001}}]\label{IteratedHahn}
Let $\k \subseteq \R\Lhp \mathfrak{N}\Rhp$ be a subfield and let $\M_1, \M_2 \subseteq\mathfrak{N}$ be subgroups such that $\M_2^{\succ 1} > \M_1^{\succ 1} > \k$. Then 
$$\k\Lhp \M_1 \Rhp \Lhp \M_2 \Rhp = \k \Lhp \M_1 \M_2 \Rhp.$$
 In particular every series $f \in \k \Lhp \M_1 \M_2 \Rhp$ has a unique representation as
\[f= \sum_{i < \alpha} k_i \mathfrak{m}_i \qquad \text{with} \qquad k_i \in \k\Lhp \M_1 \Rhp \quad \text{and} \quad \mathfrak{m}_i \in \M_2 \quad \text{for every}\; i<\alpha.\]
\end{rem}

\section{Transseries}
The field $\T$ of transseries is a subfield of a field of the form $\R((\M^{LE}))$ where $\M^{LE}$ is a suitable ordered multiplicative group called the group of transsmonomials. We shall not define $\T$, but in this and the following section we list all the properties needed in this paper. In particular we shall need the fact that $\T$ is the union of the subfields $\T_{n,\lambda}$ in Definition \ref{defn:tabular}.  This representation of $\T$ will be used to introduce the ideal of subsets of $\M^{LE}$ mentioned in the introduction. 

\begin{defn}
	Denote by $\T$ the field of LE-transseries in a formal variable $\x$ as described in \cite{DriesMM2001} and let $\M^{LE}$ be the group of LE-transserial monomials. Note that $$\T\subseteq \R\Lhp\M^{LE}\Rhp.$$ Let $$\T^\uparrow = \mathbb{R} \Lhp \M^{LE, \succ 1} \Rhp \cap \T$$ be the $\R$-vector space of the transseries whose support only contains infinite monomials and observe that $\T= \T^\uparrow \oplus \R\oplus o(1)$ where $o(1)$ is the set of infinitesimal transseries. The elements of $\T^\uparrow$ are called {\em purely infinite}. 
\end{defn}
\begin{fact}
Recall that $\T$ admits an exponential function $\exp: \T \rightarrow \T^{>0}$ making it into an elementary extension of the ordered field of real numbers with the natural exponential function. The function $\exp$ restricts to an ordered group isomorphisms
\[\exp|: \big(\T^{\uparrow}, +, 0, <\big)\simeq \big(\M^{LE},\cdot, 1, <\big)\]
and this suffices to determine $\exp$ on the whole $\T$ via the formula
\begin{equation}\label{exp}
\exp (f)= \exp(\InfP{f}) \exp (\ReP{f}) \sum_{n \in \mathbb{N}} \frac{(\EpsP{f})^{n}}{n!}
\end{equation}

The compositional inverse of $\exp$ is called logarithm, $\log: \T^{>0} \rightarrow \T$, and has an analogous piecewise characterization in terms of the multiplicative decomposition: for $f >0$ one has
\begin{equation}\label{log}
\log(f) = \log(\lm (f) ) + \log( \lc(f) ) +  \sum_{n >0} \frac{(-1)^{n+1}\varepsilon}{n} \qquad \varepsilon = \frac{f-\lt(f)}{\lt(f)}
\end{equation}
For $g>0$ we define $f^g:=\exp(g \log(f) )$.
\end{fact}

\begin{defn}[Normal form]\label{defn:normal}
	Since $\M^{LE}=\exp(\T^\uparrow)$, every element $f \in \T$ has a unique representation as $$f= \sum_{i <\alpha} r_i e^{\gamma_i}$$ where $\alpha$ is an ordinal, $r_i \in \R\setminus \{0\}$ for every $i<\alpha$, $(\gamma_i)_{i< \alpha}$ is a strictly decreasing sequence of elements of $\T^{\uparrow}$ and $e^{\gamma_i}=\exp(\gamma_i)$; we call $\sum_{i <\alpha} r_i e^{\gamma_i}$ the \emph{normal form of $f$}.
\end{defn}

\section{Stratification}
Below we work in the field $\T$ of LE-transseries in the formal variable $\x$.

\begin{defn} For $n\in \N$, let $\log_n$ be the $n$-fold composition of $\log$ and let $\exp_n$ be the $n$-fold composition of $\exp$. We extend this notation to the case $n\in \Z$ with the convention that $\exp_{n}=\log_{-n}$. For example $\log_1(\x) = \log(\x)$, $\log_0(\x) = \x$ and $\log_{-1}(\x) = \exp(\x)$. 
	Now let $\exp_\Z(\x) = \{\exp_n(\x)\mid n\in \Z\}$. For $\lambda \in \exp_\Z(\x)$ and $n\in \Z$, we define $\lambda_{n}= \exp_n(\lambda)$ so that $$\lambda_{-n} = \log_n(\lambda).$$  
\end{defn}

\begin{defn}\label{defn:tabular}
For $\lambda \in \exp_\mathbb{Z}(\x)$ and $n\in \N$ let us consider the following inductively defined subsets of $\T$:
\medskip

\begin{tabular*}{10cm}{@{\extracolsep{\fill}}clll}
	(1) & $\M_{0,\lambda}:=\lambda^\R$, & $\T_{0,\lambda} = \R\Lhp \M_{0, \lambda} \Rhp$, & $\J_{0 , \lambda}=\R\Lhp \M_{0,\lambda}^{\succ 1} \Rhp$. \tabularnewline
	(2) & $\M_{n+1,\lambda}:=e^{\J_{n,\lambda}}$, & $\T_{n+1,\lambda}=\T_{n,\lambda}\Lhp\M_{n+1,\lambda}\Rhp$, & $\J_{n+1,\lambda}:=\T_{n,\lambda}\Lhp\M_{n+1,\lambda}^{\succ1}\Rhp$.\tabularnewline
\end{tabular*}
\end{defn}
For this to be well defined one needs to observe that for each $n \in \N$  we have $\T_{n, \lambda} < \mathfrak{M}_{n+1, \lambda}^{\succ 1}$ (see Definition \ref{GenCoeffSums} and Remark \ref{IteratedHahn}) and that $\T_{n, \lambda}\Lhp \M_{n+1, \lambda}\Rhp \subseteq \T$. For the verification of these facts the reader must refer to the original definition of the LE-transseries in \cite{DriesMM2001} or to the equivalent definition in \cite[Prop. 4.12]{Berarducci2019} (see in particular \cite[Lemma 4.14]{Berarducci2019}). From \cite{DriesMM2001}  or \cite[Prop. 4.18]{Berarducci2019} it also follows that 
$$\T=\bigcup_{n,\lambda} \T_{n,\lambda}$$
where $n\in \N$ and $\lambda \in \exp_\Z(\x)$. In fact in the union it suffices to take $\lambda$ of the form $\x_{-k}$ with $k\in \N$ (rather than $k\in \Z$). This depends on the fact that  $\T_{n,\exp(\lambda)} \subseteq \T_{n+1,\lambda}$. 
%
Notice that in \cite{DriesMM2001} $\T$ and $\exp$ are defined by a simultanous induction, while in \cite{Berarducci2019} the transseries are defined as a subfield of Conway's surreal numbers $\no$ with the exponentiation coming from $\no$. For a short account of the latter approach and all the relevant definitions see also \cite{Berarducci2020}. 
\begin{defn}
Although $\T$ is not a maximal valued field, its subfields $\T_{n,\lambda}$ are maximal, indeed 
$$\T_{n,\lambda} = \R\Lhp \fN_{n,\lambda}\Rhp$$ where $\fN_{n,\lambda} = \T_{n,\lambda} \cap \M^{LE}.$ The subgroups $\fN_{n,\lambda}\subseteq \M^{LE}$ can be inductively generated as follows: 
	\begin{enumerate}
		\item $\fN_{0,\lambda} = \M_{0,\lambda}$. 
		\item $\fN_{n+1,\lambda} = \fN_{n,\lambda}\M_{n+1,\lambda}$.
	\end{enumerate}	
\end{defn}	
\begin{rem}
For $n\in \N$ we have $\M_{n+1,\lambda}^{\succ 1} > \fN_{n,\lambda}$ and a direct lexicographic product $$\fN_{n+1,\lambda} = \M_{0,\lambda} \M_{1,\lambda} \cdot \ldots \cdot \M_{n+1,\lambda},$$ so  $\R\Lhp\fN_{n,\lambda} \Rhp \Lhp \M_{n+1,\lambda}\Rhp = \R\Lhp\fN_{n+1,\lambda}\Rhp$. 
\end{rem}

\begin{defn} We define $\T_{n,\lambda}^\uparrow = \R\Lhp\fN_{n,\lambda}^{\succ 1}\Rhp$ and observe that $\T^\uparrow = \bigcup_{n,\lambda} \T_{n,\lambda}^\uparrow$. 
\end{defn}

\begin{rem}[Various normal forms]\label{rem:forms}
Since $\M_{n+1,\lambda}^{\succ 1} > \fN_{n,\lambda}$ we have $\fN_{n+1,\lambda}^{\succ 1}=\big(\fN_{n,\lambda}\M_{n+1,\lambda})^{\succ 1}= \fN_{n, \lambda}^{\succ 1} \cup (\fN_{n, \lambda} \M_{n+1,\lambda}^{\succ 1})$ and $\fN_{n, \lambda}^{\succ 1}<\fN_{n, \lambda} \M_{n+1,\lambda}^{\succ 1}$, so applying $\R \Lhp - \Rhp$ we get $\T_{n+1, \lambda}^{\uparrow} = \T_{n, \lambda}^{\uparrow} + \J_{n+1, \lambda}$ and by induction we easily obtain
	\begin{equation} \label{rem:forms:eqn1}
	\T_{n,\lambda}^\uparrow = \J_{0,\lambda} + \ldots + \J_{n, \lambda}.
	\end{equation}
It follows that $$\fN_{n+1,\lambda} = \M_{0,\lambda} \M_{1,\lambda} \cdot \ldots \cdot \M_{n+1,\lambda}= \lambda^\R\exp(\T_{n,\lambda}^\uparrow),$$ thus $\T_{n+1,\lambda} = \R\Lhp\lambda^\R \exp(\T_{n,\lambda}^\uparrow)\Rhp$. In other words every $f\in \T_{n+1,\lambda}$ can be written as 
$$f = \sum_{i<\alpha} r_i\lambda^{s_i} e^{\alpha_i}$$
 where  $\alpha$ is an ordinal, $r_i\in \R \setminus \{0\}$, $s_i\in \R, \alpha_i \in \T_{n,\lambda}^\uparrow$ and $(\lambda^{s_i} e^{\alpha_i})_{i<\alpha}$ is strictly decreasing. We can also write it in the form $$f = \sum_{i<\alpha} r_i e^{\beta_i}$$ where $\beta_i = s_i \log(\lambda) + \alpha_i$ (this is the normal form of Definition \ref{defn:normal}). 

Recalling that $\T_{n+1,\lambda} = \T_{n,\lambda}\Lhp \M_{n+1,\lambda}\Rhp = \T_{n,\lambda}\Lhp e^{\J_{n,\lambda}}\Rhp$ we also have a representation of the form 
$$f = \sum_{i<\alpha} k_i e^{\gamma_i}$$
where $k_i \in \T_{n,\lambda}$ and $\gamma_i \in \J_{n,\lambda}$. 
\end{rem}		

\begin{defn} It is convenient to extend Definition \ref{defn:tabular} to the case when $n\in \Z$, so we put
\begin{enumerate}
	\item $\M_{-1,\lambda} = 1$.
	\item $\T_{-1,\lambda} = \R$. 
	\item $\J_{-1,\lambda} = \R \log(\lambda)$. 
\end{enumerate}
and for $n<-1$ we put $\M_{n,\lambda} = 1, \T_{n,\lambda} = \R$ and $\J_{n,\lambda} = \{0\}$. 
\end{defn}
\begin{prop}\label{prop:increasing}
For all $n\in \Z$ and $\lambda\in \exp_\Z(\x)$, 
\begin{enumerate}
	\item $\M_{n+1,\lambda_{-1}} \supseteq \M_{n,\lambda}$. 
	\item $\T_{n+1,\lambda_{-1}} \supseteq \T_{n,\lambda}$.  
	\item $\J_{n+1,\lambda_{-1}} \supseteq \J_{n,\lambda}$. 
\end{enumerate}
\end{prop}
\begin{pf}
	The case $n=-1$ follows from the following inclusions:   
	\begin{itemize}
		\item $\M_{0,\lambda_{-1}} =\lambda_{-1}^\R \supseteq 1 = \M_{-1,\lambda}$;
		\item $\T_{0,\lambda_{-1}} = \R\Lhp\lambda_{-1}^\R\Rhp \supseteq \R= \T_{-1,\lambda}$;  
		\item $\J_{0,\lambda_{-1}} = \R\Lhp\lambda_{-1}^{\R^{>0}}\Rhp\supseteq \J_{-1,\lambda}$. 
	\end{itemize}	
	We can then conclude by an easy induction argument. 
\end{pf}

\section{Levels}
Our next goal is to represent the ordered vector space $\T^\uparrow$ as a lexicographic direct sum $\bigoplus_{n \in\Z} \J_n$ of suitable subspaces which can be characterized in terms of ``levels''. Since $\exp(\T^{\uparrow}) = \M^{LE}$, this will also induce a decomposition of $\M^{LE}$ as a direct sum of multiplicative subroups $\M_n = \exp(\J_{n-1})$. 

Recall that $\x$ is the formal variable of $\T$ and $\x_{-k} = \log_k (\x)$. Mapping $\x$ to $\x_{-k}$ will induce an automorphism of $\T$ sending $\J_n$ to $\J_{n-k}$ and $\M_n$ in $\M_{n-k}$. 

\begin{defn}\label{defn:Mn}
For $n\in \Z$ we define:
\begin{enumerate}
	\item $\T_n = \bigcup_{k\in \Z} \T_{n+k, \x_{-k}}$.  
	\item $\M_n = \bigcup_{k\in \Z} \M_{n+k,\x_{-k}}$. 
	\item $\fN_n = \bigcup_{k\in \Z} \fN_{n+k,\x_{-k}}$. 
	\item $\J_n = \bigcup_{k\in \Z} \J_{n+k,\x_{-k}}$. 
\end{enumerate}
\end{defn}
By Proposition \ref{prop:increasing} all the unions are increasing. 
\begin{rem} We have
 $\M^{LE} = \bigcup_{n\in \N} \fN_n$ and $\exp(\J_n) = \M_{n+1}$ by Definitions \ref{defn:Mn} and \ref{defn:tabular} 
\end{rem}

The following definion is needed to prove that the vector spaces $\J_n$ are in direct sum. 

\begin{defn}[Levels] \label{level}Let $f,g \succ 1$. 
	We say that $f$ and $g$ have the same level if there is $n\in \N$ such that $\log_n(|f|) \asymp \log_n(|g|)$. We say that $f$ has level $n\in \Z$ if it has the same level of $\exp_n(\x)$ and we write in this case $\level(f) = n$.  For $f\prec 1$, we define $\level(f) = \level(1/f)$. By convention we also stipulate that the level of an element $f \asymp 1$ is $-\infty$ and the level of $0$ is undefined. We have:
	\begin{enumerate}
		\item If $1 \not \asymp \sum_{i<\alpha} r_i e^{\gamma_i}\in \T$ is in normal form, then $\level(\sum_{i<\alpha} r_i e^{\gamma_i}) = \level({\gamma_0})+1$. 
		\item $\level(\x^s) = 1$ for all $s\in \R^*$. 
	\end{enumerate} 
\end{defn}
\begin{rem}\label{rem:convex}
The function $\level:\T^* \to \Z\cup \{-\infty\}$ satisfies: 
\begin{enumerate}
	\item $\level(fg) = \max \{\level(f), \level(g)\}$
	\item if $1 \preceq f \preceq g$ then $\level(f) \leq \level (g)$
	\item if $1\prec f , g$ and $\level(f) < \level(g)$, then $|f| < |g|$.  
\end{enumerate}
It follows that $\{\m \in \M^{LE} \mid \level (\m) \leq n\}$ is a convex subgroup of $\M^{LE}$ for all $n\in \Z$.
\end{rem}

\begin{prop} \label{prop:levels}For $n\in \Z$ we have:
	\begin{enumerate}
		\item $\J_n$ is the field of transseries whose support only consists of infinite monomials of level exactly $n$.
		\item $\fN_n = \{\m \in \M^{LE} \mid \level (\m) \leq n\}$; 
		\item $\T_n$ is the field of transseries whose support only contains monomials of level less or equal than $n$. 
	\end{enumerate}
\begin{proof}
	One easily sees by induction that if $n+k \ge 0$ and $1 \neq \m \in \M_{n+k, \x_{-k}}$ then $\level(\m)=n$.
	It follows in particular that if $\n \in \fN_{n+k,\x_{-k}}$ then $\level(\n) \le n$ and that equality holds if and only if $\n \notin \fN_{n+k-1, \x_{-k}}$. Hence the monomials of level $n$ are exactly those contained in $\M_{n}^{\not\asymp 1} \fN_{n-1}= \fN_n \setminus \fN_{n-1}$. All of (1)-(2)-(3) easily follow from this.
\end{proof}
\end{prop}

\begin{prop} \label{Dec&Char} We have $\J_{n+1}^{>0}>\J_n$ and
	$\T^\uparrow$ is the direct sum $$\T^\uparrow = \bigoplus_{n\in \Z} \J_n$$ as an ordered $\R$-vector space .
\end{prop}
\begin{pf} The inequality $\J_{n+1}^{>0}>\J_n$ follows from Remark \ref{rem:convex}(3). 
	The sum is direct by Proposition \ref{prop:levels}(1) and it is equal to $\T^\uparrow$ by Remark \ref{rem:forms}, Equation \ref{rem:forms:eqn1}.  
\end{pf}

\begin{cor}
 The group $\M^{LE}$ is the multiplicative direct sum of the subgroups $\M_n$. 
\end{cor}

\section{Summability}\label{sec:Summability}
In this section we introduce the ideal of subsets of $\M^{LE}$ mentioned in the introduction. 

\begin{defn}
	Given a family $(\m_i)_{i<\alpha}$ of monomials in $\M^{LE}$, we say that $(\m_i)_{i<\alpha}$ is {\em $\T$-summable} if $(\m_i)_{i<\alpha}$ is summable and the Hahn series $\sum_{i<\alpha} \m_i$ belongs to $\T\subseteq \R\Lhp \M^{LE}\Rhp$. 
	Note that if $(\m_i)_{i<\alpha}$ is $\T$-summable, then for every sequence of non-zero real numbers $(r_i)_{i<\alpha}$ we have $\sum_{i<\alpha} r_i \m_i \in \T$. 
\end{defn}

%
%
We give below a reformulation of $\T$-summability which is more convenient for our treatment. 

\begin{defn}\label{BDef}
	Given a set $X$, we say that $\mathcal{X}$ is a \emph{bornology} on $X$ if it is an ideal in the posets of subsets of $X$ whose union is $X$, that is:
	\begin{enumerate}
		\item whenever $Z\subseteq Y \in \mathcal{X}$ one has $Z \in \mathcal{X}$  
		\item whenever $Z, Y \in \mathcal{X}$ one has $Z\cup Y \in \mathcal{X}$
		\item $\bigcup \mathcal{X} = X$.
	\end{enumerate}
	
	Given $X_\mathcal{X}= (X, \mathcal{X})$ we say that a subset $S \subseteq X$ is $\mathcal{X}$-bounded if it is a subset of some element of $\mathcal{X}$.
	
	A map between sets endowed with a bornology $f: X_{\mathcal{X}} \rightarrow Y_{\mathcal{Y}}$ is said to be \emph{bounded} if the image of any bounded subset is a bounded subset. A bijection is said to be \emph{bi-bounded} if it is bounded with bounded inverse.
	
	Let $Y \subseteq X$ and $\mathcal{X}$ a bornology on $X$, then $Y$ naturally carries a bornology $$\mathcal{X}|_Y=\{Z\cap Y: Z \in \mathcal{X}\}$$
	consisting of those subsets that are $\mathcal{X}$-bounded when regarded as subsets of $X$. It is the largest bornology making the inclusion a bounded map.
	
	Given a set $X$ and a family of subsets $\mathcal{F}$ whose union is $X$, the smallest bornology containing $\mathcal{F}$ is said to be the bornology \emph{generated} by $\mathcal{F}$. If $\mathcal{F}$ is upward directed, the bornology generated by $\mathcal{F}$ is the family of subsets of $X$ that are contained in some $F \in \mathcal{F}$.
\end{defn}

\begin{defn}\label{defn:bornologies} We introduce the following bornologies on $\M^{LE}, \T, \T_0[t^{\pm 1}]$. 
	\begin{enumerate}
		\item $\mathcal{M}$ is the bornology on $\M^{LE}$ generated by  the subgroups $\mathfrak{N}_{n, \x_{-k}}$ for $n,k \in \N$
		\item $\mathcal{T}$ is the bornology on $\T$ generated by the subfields $\T_{n,\x_{-k}}$ for $n,k\in \N$. 
		\item We also consider on $\T_0[t^{\pm 1}]$ the bornology generated by the subgroups $E[t^{[m,n]}]= t^{n}E +\cdots + t^{m}E$ where $E$ is a $\mathcal{T}$-bounded subfield of $\T_0$ and $m\leq n$ are in $\Z$. 
	\end{enumerate} 
\end{defn}

\begin{rem}
	With the above definition it follows that 
	a set of monomials $S\subseteq \M^{LE}$ is $\T$-summable if and only if it is reverse well ordered and $\mathcal{M}$-bounded. 
\end{rem}

\begin{rem}\label{ExpBounded}
	Note that $\mathcal{M}= \mathcal{T}|_{\M^{LE}}$ and that
	that $\exp$ and $\log$ are bi-bounded maps with respect to the bornologies we just introduced as $\exp( \T_{n, \x_{-k}}) \subseteq \T_{n+1, \x_{-k}}$ and $\log( \T_{n, \x_{-k}}^{>0} ) \subseteq \T_{n, \x_{-k-1}}$.
\end{rem}

\section{The crucial isomorphism}
We recall that $\T_n$ is the field of transseries whose support only contains monomials of level less or equal than $n$ and $\J_n$ is the field of transseries whose support only consists of infinite monomials of level exactly $n$. We shall prove that there is an isomorphism $\T_n \cong \J_n$ of ordered vector spaces. 

\begin{prop}\label{PureIso}
For each $n\in \Z$, there is an isomorphism of ordered $\R$-vector spaces $f_n: \T_n\simeq \J_n$. Moreover the isomorphism maps $\T_{n+k, \x_{-k}}\subseteq \T_n$ onto $\J_{n+k, \x_{-k}}\subseteq \J_n$ for $ n+k \ge-1$ so it is bi-bounded with respect to the bornologies $\mathcal{T}|_{\T_n}$ and $\mathcal{T}|_{\J_n}$.
\end{prop}
\begin{proof}
	\newcommand{\Isoto}{\overset{\sim}{\longrightarrow}}
	It suffices to show that for each $k,n\in \Z$ such that $n+k\ge -1$ there is an isomorphism $f_{n+k,\x_{-k}}: \T_{n+k,\x_{-k}}\simeq \J_{n+k,\x_{-k}}$ and that these isomorphisms can be glued together so to define the $f_n$ s as
	$$f_n:= \bigcup_{k \ge -n-1} f_{n+k, \x_{-k}} : \T_{n} \longrightarrow \J_{n}.$$
	Easing the notation, we shall consider isomorphisms
	$$f_{n, \lambda} : \T_{n,\lambda} \Isoto \J_{n, \lambda} $$
	for any $\lambda \in \log_{\mathbb{Z}}(\x)$ and $n \ge -1$, and we define them by induction on $n$ starting at $n=-1$.
	 
	Note that the relation $\T_{n+1, \lambda}=\T_{n, \lambda}\Lhp e^{\J_{n,\lambda}}\Rhp$ holds for every $n\ge -1$.
	 
	First let $h: \T \rightarrow \T$ be defined as 
	\begin{equation}\label{hIso}
	h(x)= \begin{cases}
		x+1 &\text{if}\; x\ge 0\\
		\frac{1}{1-x} &\text{if}\; x \le 0.
		\end{cases}
	\end{equation}
	We shall use only the fact that $h$ is an order isomorphism $\T \simeq \T^{>0}$ mapping $0$ to $1$ and restricting to $h|: \T_{n, \lambda} \Isoto \T_{n, \lambda}^{>0}$ for every $n\in \mathbb{Z}$ and $\lambda \in \exp_\Z(\x)$.
	
	We build inductively $f_{n, \lambda}$ as follows: for $n=-1$ we just set
	$$f_{-1, \lambda} : \T_{-1,\lambda}=\mathbb{R} \Isoto \J_{-1, \lambda}=\log(\lambda)\mathbb{R} \qquad f_{-1, \lambda}(r)= \log(\lambda) r$$
	Then to define $f_{n+1,\lambda}$ from $f_{n,\lambda}$ we use 
	$$\T_{n+1,\lambda}= \T_{n,\lambda}\Lhp \exp(\J_{n,\lambda})\Rhp \quad \quad \text{and} \quad \quad \J_{n+1,\lambda} = \T_{n,\lambda}\Lhp \exp(\J^{>0}_{n,\lambda})\Rhp.$$
	Assuming inductively that we have an isomorphism
	$$f_{n, \lambda} : \T_{n,\lambda} \Isoto \J_{n,\lambda}$$ 
	composing with the order isomorphism $h|_{\T_{n,\lambda}}:\T_{n,\lambda}\simeq \T_{n,\lambda}^{>0}$ we obtain an induced order isomorphism
	$$f_{n, \lambda} \circ h\circ f_{n, \lambda}^{-1} : \J_{n,\lambda} \Isoto \J^{>0}_{n,\lambda},$$ which in turn induces an isomorphism 
	$$f_{n+1, \lambda} : \T_{n+1,\lambda} \Isoto \J_{n+1,\lambda}$$
	as follows: every element $x$ of $\T_{n+1, \lambda}$ may be written uniquely as
	\[x= \sum_{i< \alpha} k_i \exp(\gamma_i) \qquad k_i \in \T_{n,\lambda}\setminus \{0\} \qquad \gamma_i \in \J_{n, \lambda}\]
	then one sends $x$ to
	\[ f_{n+1, \lambda}(x) = \sum_{i< \alpha} k_i \exp\big(f_{n, \lambda} \circ h \circ f_{n, \lambda}^{-1} (\gamma_i)\big)\] 
	This way we end up with a family of isomorphisms $f_{n, \x_{-k}}: \T_{n, \x_{-k}} \Isoto \J_{n, \x_{-k}}$.
	
	The glueing needed then follows from the following claim:\smallskip
	
	\centerline{\textit{for every $\lambda = \x_{-k}$ and every $n\ge -1$, $f_{n+1, \log(\lambda)}$ extends $f_{n, \lambda}$}.}\smallskip
	To prove this claim we proceed by induction on $n$. For $n=-1$ we easily see that for every $r \in \mathbb{R}$ we have 
	\[\begin{split}
	f_{0, \log(\lambda)}(r)
	&=  r \exp(f_{-1,\log(\lambda)} \circ h \circ f_{-1, \log(\lambda)}^{-1}(0\Rhp =\\
	&=  r \exp(f_{-1,\log(\lambda)} \circ h (0\Rhp =\\
	&=  r \exp(f_{-1,\log(\lambda)} (1) ) = r \log(\lambda) = f_{-1, \lambda}(r).
	\end{split}\]
	
	As for the inductive case, assume $f_{n+1, \log(\lambda)}$ extends $f_{n, \lambda}$, and let us prove that $f_{n+2, \log(\lambda)}$ extends $f_{n+1, \lambda}$: let $x \in \T_{n+1, \lambda}\subseteq \T_{n+2, \log(\lambda)}$, the crucial observation is that if we write $x$ as
	\[ x= \sum_{i< \alpha} k_i \exp(\gamma_i) \qquad k_i \in \T_{n+1,\log(\lambda)}\setminus \{0\} \qquad \gamma_i \in \J_{n+1, \log(\lambda)}\]
	since $x \in\T_{n+1, \lambda}$ one has that actually $k_i \in \T_{n, \lambda}$ and $\gamma_i \in \J_{n, \lambda}$. Hence  recalling the inductive hypothesis we have
	\begin{align*}
	f_{n+2, \log(\lambda)}(x)= \sum_{i< \alpha} k_i \exp\big(f_{n+1, \log(\lambda)} \circ h \circ f_{n+1, \log(\lambda)}^{-1} (\gamma_i)\big)=\\
	=  \sum_{i< \alpha} k_i \exp\big(f_{n,\lambda} \circ h \circ f_{n, \lambda}^{-1} (\gamma_i)\big)= f_{n+1,\lambda}(x).
	 \end{align*}
	 This completes the proof.
\end{proof}	

\begin{rem}
Note that by construction $f_n$ is strognly $\mathbb{T}_{n-1}$-linear. More preciesly, given a transseries of the form $\sum k_i e^{\gamma_i}$ where $k_i \in \T_{n-1}$ and $\gamma_i \in \J_{n-1}$, we have that
$$f_n\left(\sum k_i e^{\gamma_i}\right) = \sum k_i e^{f_{n-1} \circ h \circ f_{n-1}^{-1}(\gamma_i)}.$$
This can be used to compute $f_n$ using as a base case $f_n\rest \R$ which is given by $f_n(r) = r \x_n$. In particular $f_0(r) = r\x$. Let us also note that $f_n\rest \T_{0,\x_n}$ is given by
$f_n(\sum_{i<\alpha} r_i \x_n^{s_i}) = \sum_{i<\alpha} r_i \x_n^{h(s_i)}$. 
\end{rem}

\begin{example} Consider the transseries $\exp(\x e^{-\x})$. Its normal form is 
$$\exp(\x e^{-\x})=\sum_{n\in\mathbb{N}} \frac{\x^n e^{-n\x}}{n!} \in \mathbb{T}_{1, \x}\subseteq \T_{2, \x}$$
We compute $f_1$ and $f_2$ on $\exp(\x e^{-\x})$. 
\begin{itemize}
\item $\displaystyle{f_{1}(\exp(\x e^{-\x}\Rhp= \sum_{n\ge 0} \frac{\x^n e^{f_0 \circ h \circ f_0^{-1}(-n\x)}}{n!}=\sum_{n\ge 0} \frac{\x^n e^{\frac{1}{n+1} \x}}{n!}} \in \J_{1, \x}$, because\\
$f_0 \circ h \circ f_0^{-1}(-n\x) = f_0 \circ h (- n ) = \frac{1}{n+1}\x$. 
\smallskip
\item $\displaystyle{
	f_{2}(\exp(\x e^{-\x}\Rhp = \exp(\x e^{-\x}) e^{f_1 \circ h \circ f_1^{-1}(0)} = \exp(\x e^{-\x}) e^{e^{\x}}\in \J_{2,\x}
}$ 
\end{itemize}
\end{example}

\section{The group of monomials is order isomorphic to its positive cone}
In this section we show that there is a bi-bounded order isomorphism $\M^{LE} \cong \M^{LE,\prec 1}$. Since $\exp(\T^\uparrow) = \M^{LE}$ and $\exp$ is bi-bounded (Remark \ref{ExpBounded}) this reduces to show that there is a bi-bounded order isomorphism $T^{\uparrow}\cong \T^{\uparrow,>0}$. In turn this depends on the fact that $\T^{\uparrow}$ is isomorphic to the ordered vector space $\T_0[t^{\pm 1}]$ of Laurent polynomials over the field $\T_0$. In fact we will show that, for any ordered field $K$, the vector space $K[t^{\pm 1}]$ is order isomophic to its positive cone. Since all the relevant isomorphisms are bi-bounded with respect to the appropriate bornologies, combining the isomorphisms we obtain our main result. 

\begin{rem}
Notice that for all $n\in \Z$ there is an automorphism $S_n$ of $\T$ preserving $\exp$ and infinite sums and sending $\x$ to $\x_n= \exp_n(\x)$ (see \cite{DriesMM2001}). The restriction of $S_n$ is an isomorphism $S_n| : (\T_n, +, \cdot, 0,1 ,<) \simeq (\T_0, +, \cdot, 0,1 ,<)$. 	
\end{rem}
 
\begin{prop}\label{Prop:LaurentIso}
There is a bi-bounded isomorphisms of ordered $\R$-vector spaces $$F: \T_0[t^{\pm 1}]\to \T^\uparrow.$$
where $\T_0[t^{\pm1}]$ is the ordered ring of Laurent polynomials with coefficients from $\T_0$ ordered with the condition $t>\T_0$ and is endowed with the bornology defined in Proposition~\ref{defn:bornologies}. 
\end{prop}
\begin{pf} It suffices to compose the isomorphisms 
	$$	 
	\T_0[t^{\pm1}]
	\simeq 
	\bigoplus_{n \in\Z} \T_0 \simeq
	\bigoplus_{n \in\Z} \T_n \simeq
	\bigoplus_{n \in\Z} 
	\J_n
	\simeq
	\T^\uparrow
	$$
	More precisely, 
	let $f_n:\T_n \to \J_n$ be as in Proposition~\ref{PureIso}, then one defines
	$$F: \T_0[t^{\pm1}] \rightarrow \T^\uparrow, \qquad F(\sum k_i t^i)= \sum f_n \circ S_n(k_n).$$
	This is an ordered isomorphism by virtue of Proposition~\ref{Dec&Char}.
	
	Notice that $F(t^i \T_{n,\x_{-n-i}}) = \J_{n,\x_{-n+i}}$, hence by linearity
	$$F(\T_{n,\x_{-n}}[t^{[-m,m}]) = \sum_{i\in [-m,m]} \J_{n,\x_{-n+i}}.$$
	%
	%
	To prove that $F$ is bi-bounded it suffices to show that,
	as $m\leq n$ range in $\Z$, the sets $\T_{n,\x_{-n}}[t^{[-m,m}]$ and $\sum_{i\in [-m,m]} \J_{n,\x_{-n+i}}$ generate the bornologies of $\T_0[t^{\pm 1}]$ and $\T^\uparrow$ respectively. This is clear for $\T_0[t^{\pm 1}]$. Now recall that the bornology of $\T^\uparrow$ is generated by the sets $\T^\uparrow_{n,\x_{-k}}$ and, by Remark~\ref{rem:forms}, Equation~\ref{rem:forms:eqn1}, 
	for $n \in \N$ and $k\in \Z$, setting $m=\max\{|n-k|, |k|\}$, we get
	$$\T_{n, \x_{-k}}^\uparrow \subseteq \J_{n, \x_{-n-m}} + \cdots + \J_{n, \x_{-n+m}} \subseteq \T_{n+2m, \x_{-m-n}}^\uparrow.$$
	
\end{pf}

\begin{prop}\label{Laurent} 
	Given an ordered field $K$ and a bornology on $\mathcal{K}$ on $K$ generated by subfields, consider the
	bornology $\mathcal{K}[t]$ generated by subgroups of the form $E[t^{[m,n]}]=Et^m + \cdots +Et^n$ as $E$ ranges in the $\mathcal{K}$-bounded subfields of $K$ and $m\le n$ range in $\Z$. 
	There are: 
	\begin{enumerate}
		\item a bi-bounded order isomorphism $h: K\simeq K^{>0}$
		\item a bi-bounded order isomorphism $\mathcal{H}: K[t^{\pm 1}] \simeq K[t^{\pm 1}]^{>0}$ extending $h$ where $K[t^{\pm 1}]$ is the additive group of all Laurent polynomials, with the ring order induced by $t>K$.
	\end{enumerate}
	\begin{proof}
		
		(1) An order isomorphism can be defined piecewise, e.g. setting
		\[h(x)= \begin{cases}
		x+1 &\text{if}\; x\ge 0\\
		\frac{1}{1-x} &\text{if}\; x \le 0.
		\end{cases}
		\]
		One easily sees that since $h$ is defined only in terms of the order and of field operations and constants, for every subfield $L \subseteq K$ it restricts to $h|: L \simeq L^{>0}$, hence it is bi-bounded. 
		\medskip
		
		(2) This is a bit more involved. 
		We will define order isomorphisms $A$ and $B$ as in the diagram below
		\begin{center}
			$\begin{tikzcd}
			K[t^{\pm 1}] \arrow[r, "A"] \arrow[dr, "\mathcal{H}"]
			& \Z \overset{>}\times K[t^{-1}]    \\
			&K[t^{\pm 1}]^{>0} \arrow[u, "B"]\end{tikzcd}$
		\end{center}
		and define $\mathcal{H}$ as the composition $\mathcal{H}:= B^{-1} \circ A$.
		Here $\mathbb{Z} \overset{>}{\times} K[t^{-1}]$ denotes the product of $\Z$ and $K[t^{-1}]$ endowed with the lexicographic total order, that is $(x,y) < (x', y')$ if and only if $x<x'$ or $x=x'$ and $y<y'$.
	
		In order to define the isomorphisms $A$ let us observe that the set $t^{n}K[t^{-1}]$ of Laurent polynomials of degree $\leq n$ can be partitioned into three order-convex subsets $$t^{n}K[t^{-1}] \;=\; L_n \;\cup\;  t^{n-1} K[t^{-1}] \;\cup\; U_n$$
		with
		$$L_n\;<\; t^{n-1} K[t^{-1}] \;<\; U_n$$
		where $L_n$ is the set of negative Laurent polynomials of degree $n$ and $U_n$ is the set of positive Laurent polynomials of degree $n$, that is: 
		\begin{align*}
		&L_n:= t^{n-1} K[t^{-1}] + K^{<0} t^{n} \\
		&U_n:= t^{n-1} K[t^{-1}] + K^{>0} t^{n}. 
		\end{align*}
		
		It follows that we can write $K[t^{\pm1}]$ as the disjoint union
		$$ K[t^{\pm1}] = \bigcup_{n >0} L_n \cup K[t^{-1}] \cup \bigcup_{n >0} U_n.$$
 		By point (1) we have order isomorphisms $K \simeq K^{>0}\simeq K^{<0}$. It follows that we can write down induced order isomorphisms 
		\begin{align*}
		u_n : U_n &\rightarrow \{n\} \times K[t^{-1}], \qquad &&y+xt^{n} &\mapsto&\; \big(n, t^{-n} y +h^{-1}(x) \big)\\
		l_{n} : L_n &\rightarrow \{-n\} \times K[t^{-1}], \qquad &&y -xt^{n} &\mapsto&\; \big(-n, t^{-n} y-h^{-1}(x) \big)
		\end{align*}
		where $y$ ranges in $t^{n-1}K[t^{-1}]$ and $x$ in $K^{>0}$.
		
		Thus we can define an order isomorphism $A: K[t^{\pm1}] \rightarrow \Z \overset{>}{\times} K[t^{-1}]$ as the union
		$$A= \bigcup_{n>0} l_n \cup A_0 \cup \bigcup_{n>0} u_n : K[t^{\pm1}] \rightarrow \mathbb{Z} \overset{>}{\times} K[t^{-1}]$$
		where $A_0: K[t^{-1}] \simeq  \{0\} \times K[t^{-1}]$ is the obvious isomorphism $y \mapsto (0 ,y)$.
		
		\smallskip
		Similarly $K[t^{\pm1}]^{>0}$ decomposes as a disjoint union
		\[K[t^{\pm1}]^{>0} = \bigcup_{n \in \Z} U_n \qquad \text{with}\qquad U_n<U_{n+1}, \]
		and again we can define the order isomorphism $B : K[t^{\pm1}]^{>0} \rightarrow \Z \overset{>}{\times} K[t^{-1}]$ as the union for $n \in \Z$ of $u_n: U_n \simeq \{n\} \times K[t^{-1}]$:
		\[B = \bigcup_{n \in \Z} u_n : K[t^{\pm1}]^{>0} \rightarrow \mathbb{Z} \overset{>}{\times} K[t^{-1}].\]

		In order to prove that $\mathcal{H}: B^{-1}\circ A: K[t^{\pm1}] \to K[t^{\pm1}]^{>0}$ is bi-bounded it suffices to prove that for any subfield $E$ of $K$ we have
		$$E[t^{[-n,n]}]^{>0} \subseteq \mathcal{H} (E[t^{[-n,n]}]) \subseteq E[t^{[-2n, 2n]}]^{>0}.$$
		
		Computing $A$ and $B$ on $E[t^{[-n,n]}]$, where $E$ is a subfield of $K$, we have
		$$ A(E[t^{[-n,n]}]) =
		\bigcup_{|k| \le n} \{k\} \times E[t^{[-|k|-n,0]}],$$
		$$ B(E[t^{[-n,n]}]) =
		\bigcup_{|k| \le n} \{k\} \times E[t^{[-k-n,0]}].
		$$
		From this we see that clearly $B(E[t^{[-n,n]}]) \subseteq A( E[t^{[-n,n]}])  \subseteq B(E[t^{[-2n,2n]}])$. The claim follows applying $B^{-1}$. 
	\end{proof}
\end{prop}

\begin{defn}\label{def:Eta} 
	By Proposition~\ref{Laurent} point (2), setting $K= \T_0$, $\mathcal{K}= \mathcal{T}|_{\T_0}$, we get an order isomorphism  $$\mathcal H : \T_{0}[t^{\pm1}] \rightarrow \T_{0}[t^{\pm1}]^{>0}$$
	which is bi-bounded with respect to the bornology of Definition~\ref{defn:bornologies}.
\end{defn}
 \begin{defn}
 	Let $F: \T_{0}[t^{\pm1}] \rightarrow \T^{\uparrow}$ be the ordered $\R$-vector space isomorphism described in Proposition~\ref{Prop:LaurentIso}. We define $\overline{\mathcal{H}}$ and $\eta$ as follows:
 	$$\begin{aligned}
 	\overline{\mathcal{H}} = F \circ \mathcal{H} \circ F^{-1} &: \T^{\uparrow} \rightarrow \T^{\uparrow, >0}\\
 	\eta=\exp \circ \overline{\mathcal{H}} \circ  \log &: \M^{LE} \rightarrow \M^{LE, \succ 1}.
 	\end{aligned}$$
 \end{defn}

\begin{lemma}\label{Important1} The map $$\eta: (\M^{LE}, <) \simeq (\M^{LE, \succ 1}, <)$$ is a bi-bounded order isomorphism. 
\end{lemma}
\begin{proof}
	We know that all the maps $\exp, \mathcal{H}, F$ are bi-bounded order isomorphisms hence $\eta$, being a composition of them and their inverses has to be a bi-bounded isomorphism (see Remark~\ref{ExpBounded}, Proposition~\ref{Prop:LaurentIso} and Proposition~\ref{Laurent}).
\end{proof}	

\begin{thm}\label{thm:Main}
	The ordered group of transserial monomials $\M^{LE}$ is isomorphic to the ordered additive reduct of $\T$.
\end{thm}

\begin{proof}
	Consider the isomorphism of ordered $\R$-vectors spaces 
	 $$H: \T \rightarrow \T^\uparrow$$ defined by 
	$H(\sum r_i \m_i )= \sum r_i \eta(\m_i)$.  Note that $H$ is well defined because $\eta$ is bounded and it is an isomorphism because its inverse is bounded. Now define $$\Omega: (\T, 0,+, <) \rightarrow (\M^{LE}, 1, \cdot, <)$$ as the composition 
	$\Omega=\exp \circ H$. Then $\Omega$ is the desired isomorphism. 
\end{proof}

\section{Generalizing a result on omega-maps}
In \cite{Berarducci2018b} it is shown that a field of the form $\R((\M))_{<\kappa}$ admits an omega-map if and only if admits an exponential and $\M$ is order isomorphic to $\M^{>1}$. We generalize this result to the case when instead of $\R((\M))_{<\kappa}$ we have a subfield of $\R((\M))$ induced by a bornology on $\M$. 

\begin{defn}\label{defn:BField}
	Given an ordered group $(\M, \cdot, 1, <)$, if $\Gamma$ is a subset of $\M$ and $\mathcal{G}$ is a bornology on $\Gamma$ we define 
	\[\R \Lhp \Gamma_{\mathcal{G}} \Rhp := \bigcup_{S \in \mathcal{G}} \R \Lhp S\Rhp, \]
	that is, the subspace of $\R\Lhp \Gamma \Rhp$ consisting of well founded sums with $\mathcal{G}$-bounded support.
\end{defn}

\begin{rem}\label{Subgroups&Exp}
	Let $(\M,\cdot, 1, <)$ be an ordered abelian group and let $\mathcal{M}$ be a bornology generated by subgroups. Then:
	\begin{itemize}
		\item $\R\Lhp \M_{\mathcal{M}} \Rhp \subseteq \R \Lhp \M \Rhp$ is a subfield (as it is a directed union of fields).
		\item if $\varepsilon \in \R\Lhp \M_\mathcal{M} \Rhp $ is infinitesimal and $(k_n)_{n \in \mathbb{N}}$ is a $\mathbb{N}$-sequence in $\R$, then
		$$\sum_{n \in \mathbb{N}} k_n \varepsilon ^n \in \R\Lhp \M_\mathcal{M}\Rhp.$$
	\end{itemize}
	It follows in particular that in order for fields of the form $\R\Lhp \M_{\mathcal{M}} \Rhp$ to have an exponential it suffices that they have one restricted to purely infinite elements, the extension being constructed as in Equation~\ref{exp}. 
\end{rem}

\begin{example}
	Let $(\M, \cdot, 1, <)$ be a multiplicatively written ordered abelian group.
	\begin{enumerate}
		\item If $\kappa$ is an uncountable regular cardinal, the family $\mathcal{M}_\kappa$ of subsets of $\M$ having cardinality strictly less than $\kappa$ is a bornology which can be generated by subgroups. The field $\R\Lhp \M_{\mathcal{M}_{\kappa}} \Rhp$ is the \emph{field of $\kappa$-bounded Hahn-series} and is also denoted by $\R\Lhp \M \Rhp_{\kappa}$.
		\item The family $\grid$ of subsets contained in some finitely generated subgroups of $\mathfrak{M}$ is the smallest bornology on $\M$ generated by subgroup.
		The field $\R\Lhp \mathfrak{M}_\grid \Rhp$ is called \emph{field of grid based} series (cfr \cite{VanderHoeven2006}).
		\item Taking $\M= \x^{\Q}$ we obtain the field of Puiseux series $\R\Lhp x^\Q_\grid \Rhp$. 
		\item The field $\T$ of LE-transseries coincides with $\R\Lhp\M^{LE}_\mathcal{M}\Rhp$ where $\mathcal{M}$ is the bornology in Definition \ref{defn:bornologies}. 
		
	\end{enumerate}
\end{example}

\begin{rem}\label{BMonInduced}
	If $f:\Gamma_\mathcal{G} \rightarrow \Delta_{\mathcal{D}}$ is a bounded increasing map between total orders, then the natural induced map $F: \R \Lhp \Gamma\Rhp \rightarrow \R \Lhp \Delta \Rhp$ defined as
	\[F \sum_{i < \alpha} k_i \gamma_i = \sum_{i <\alpha} k_i f(\gamma_i)\]
	maps $\R \Lhp \Gamma_{\mathcal{G}} \Rhp$ into $\R \Lhp \Delta_{\mathcal{D}} \Rhp$.
\end{rem}

The following results generalizes Theorem~4.1 of \cite{Berarducci2018b}.

\begin{prop}\label{Omega&Exp}
	Let $\fN$ be a multiplicatively written ordered abelian group and let $\mathcal{N}$ be a bornology generated by subgroups of $\fN$. For the field $\K= \R\Lhp \fN_{\mathcal{N}} \Rhp$, denote by $\mathcal{K}$ the bornology on $\K$ generated by the subfields of the form $\R\Lhp \fN'\Rhp$ as $\fN'$ ranges in the $\mathcal{N}$-bounded subgroups of $\fN$, then the following are equivalent:
	\begin{enumerate}
		\item there is an ordered bi-bounded isomorphism $(\K, +, 0, <, \mathcal{K}) \simeq (\mathfrak{N}, \cdot, 1 ,<, \mathcal{N})$ between the field and its group of values $v(\K)=\mathfrak{N}$; 
		\item $\K$ admits a bi-bounded isomorphism $\exp| : (\K^\uparrow,+, 0,<, \mathcal{K}|) \simeq (\mathfrak{N}, \cdot, 1, <, \mathcal{N})$ and an ordered bi-bounded map $(\mathfrak{N}, \mathcal{N}, < ) \simeq (\mathfrak{N}^{\succ 1}, \mathcal{N}|, <)$
	\end{enumerate}
	where $\mathcal{N}|$ and $\mathcal{K}|$ denote the restrictions of the bornologies $\mathcal{N}$ and $\mathcal{K}$ to $\fN^{\succ 1}$ and $\K^\uparrow$ respectively.
	\begin{proof}
		Assume (1) holds, that is we have an isomorphism $\Omega : (\K, +, 0, <, \mathcal{K}) \simeq (\mathfrak{N}, \cdot, 1 ,<, \mathcal{N})$ and let $h: \K \rightarrow \K^{>0}$ be defined by the formula~\ref{hIso} (Proof of Proposition~\ref{PureIso}): it is easy to check that then 
		\[\Omega \circ h \circ  \Omega^{-1} : (\mathfrak{N}, \cdot, 1, <, \mathcal{N}) \simeq  (\mathfrak{N}^{\succ 1}, \cdot, 1, <, \mathcal{N}|)\]
		is a bi-bounded chain isomorphism, hence it allows us to define an isomorphism $G: (\K^{\uparrow}, +, 0, <) \rightarrow (\K, +, 0, <)$, as the only strongly $\R$-linear map that restricts to $\Omega \circ h^{-1} \circ \Omega ^{-1}: \mathfrak{N}^{\succ 1} \rightarrow \mathfrak{N}$ (see Remark~\ref{BMonInduced}). $G$ is then bi-bounded w.r.t. to $\mathcal{K}$ and $\mathcal{K}|$. The composite $\Omega \circ G$ is the sought exponential restricted to purely infinite elments $\Omega \circ G= \exp|: (\K^\uparrow,+, 0,<) \simeq (\mathfrak{N}, \cdot, 1, <)$: it is bi-bounded because it is a composition of bi-bounded isomorphisms.\\
		On the other hand assuming (2) if we have an isomorphism $\eta: (\mathfrak{N}, \mathcal{N}, < ) \simeq (\mathfrak{N}^{\succ 1}, \mathcal{N}|, <)$ we can immediately define an isomorphism $H: (\K, +, 0, <, \mathcal{K}) \rightarrow (\K^{\uparrow}, +, 0, <, \mathcal{K}|)$ as the only strongly $\R$-linear map restricting to $\eta$ and set $\Omega = \exp \circ H$: it is again bi-bounded because it is a composition of bi-bounded maps.
	\end{proof}
\end{prop}

\begin{rem}\label{UnrestrictedAnalticExp}
	By Remark~\ref{Subgroups&Exp}, the condition $(2)$ of Proposition~\ref{Omega&Exp} is equivalent to $\K$ admitting a surjective exponential \emph{which restricts} to a bi-bounded isomorphism $\exp| : (\K^\uparrow,+, 0,<, \mathcal{K}|) \simeq (\mathfrak{N}, \cdot, 1, <, \mathcal{N})$ and a bi-bounded order isomorphism $(\mathfrak{N}, \mathcal{N}, <) \simeq (\mathfrak{N}^{\succ 1}, \mathcal{N}|, <)$.
\end{rem}

\begin{rem}
	If we consider the case of a group $\mathfrak{N}$ endowed with the ideal of subgroups $\mathcal{N}_\kappa$ consisting of the subgroups with cardinality strictly less than $\kappa$ for some fixed regular uncountable cardinal $\kappa$, then Proposition~\ref{Omega&Exp} and Remark~\ref{UnrestrictedAnalticExp} tell us that $\K= \R \Lhp \mathfrak{N} \Rhp_{\kappa}$ admits an isomorphism $\Omega: (\K, +,0,<) \simeq (\mathfrak{N}, \cdot, 1, <)$ with its archimedean value group $v(\K)\simeq \mathfrak{N}$ if and only if $\mathfrak{N} \simeq \mathfrak{N}^{\succ 1}$ and $\K$ admits a surjective exponential such that $\exp(\K^\uparrow) = \mathfrak{N}$. This implies Theorem~4.1 of \cite{Berarducci2018b} (see also Theorem~3.4 therein). 
\end{rem}

	


\end{document}